\documentclass[reqno]{amsart}
\usepackage[utf8]{inputenc}

\usepackage{amssymb}
\usepackage{amsmath}
\usepackage{mathtools}
\usepackage{amsthm}
\usepackage{mathrsfs}
\usepackage{emptypage}
\usepackage{array}

\usepackage[
  hmarginratio={1:1},
  vmarginratio={1:1},
  heightrounded,
  textheight=600pt
]{geometry}

\usepackage{diagbox}

\usepackage{tikz}
\usetikzlibrary{matrix}
\usetikzlibrary{arrows.meta}

\newtheorem{theorem}{Theorem}[section]
\newtheorem{pr}[theorem]{Proposition}

\newtheorem{co}[theorem]{Corollary}
\newtheorem{lm}[theorem]{Lemma}

\theoremstyle{definition}
\newtheorem{remark}[theorem]{Remark}

\theoremstyle{definition}

\numberwithin{equation}{section}

\newcommand{\Pic}[1]{{\rm Pic}(#1)}
\newcommand{\Nef}[1]{{\rm Nef}(#1)}
\newcommand{\proj}[1]{\mathbb P^{#1}}
\newcommand{\Proj}{\mathbb P}
\newcommand{\bigO}[1]{\mathcal O_{#1}}
\newcommand{\normal}[2]{\mathcal N_{#1 / #2}}

\newcommand{\hodge}[3]{h^{#1,#2} (#3)}

\newcommand{\NE}[1]{{\rm NE}(#1)}

\title{Fano 4-folds having a prime divisor of Picard number 1}
\author{Saverio Andrea Secci}
\date{}

\address{Università di Torino, Dipartimento di Matematica, via Carlo Alberto 10, 10123 Torino - Italy}
\email{saverioandrea.secci@unito.it}

\begin{document}

\begin{abstract}
We prove a classification result for smooth complex Fano fourfolds of Picard number 3 having a prime divisor of Picard number 1, after a characterisation result in arbitrary dimension by Casagrande and Druel \cite{CD}. These varieties are obtained by blowing-up a $\mathbb P^1$-bundle over a smooth Fano variety of Picard number 1 along a codimension 2 subvariety. We study in detail the case of dimension 4, and show that they form 28 families. We compute the main numerical invariants, determine the base locus of the anticanonical system, and study their deformations to give an upper bound to the dimension of the base of the Kuranishi family of a general member.
\end{abstract}

\maketitle

\section{Introduction}
Fano varieties (over $\mathbb C$) have been thoroughly classified up to dimension 3, and are widely studied also in higher dimensions. They play an important role in the Minimal Model Program, since they appear as the \emph{general fibres of Mori fibre spaces}, one of the possible outcomes in the run of the Minimal Model Program. Although many results on the structure of Fano varieties have already been proven, we are still lacking a complete classification of Fano varieties starting from dimension 4.

Let $X$ be a smooth projective variety of dimension $n$. We denote by $\mathcal N_1(X)$ the real vector space of one-cycles with real coefficients, modulo numerical equivalence. Its dimension is the Picard number $\rho_X$.  Recall that $X$ is Fano if its anticanonical divisor $-K_X$ is ample.

In this paper we focus on Fano varieties of dimension $n \geq 3$, admitting a prime divisor $D$ such that $\rho_D=1$: in \cite[Proposition 5]{Tsu}, Tsukioka shows that if this condition is satisfied, then the Picard number of $X$ is at most 3. Later on, these varieties have been studied also by Casagrande and Druel in \cite{CD}. To be more precise, for any prime divisor $i \colon D \hookrightarrow X$, we denote by $\mathcal N_1(D,X)$ the image of the push-forward $i_* \colon \mathcal N_1(D) \to \mathcal N_1(X)$ induced by the inclusion, which is generated by the classes in $\mathcal N_1(X)$ of curves contained in $D$. Note that $0 < \dim \mathcal N_1(D,X) \leq \rho_D$. Let us assume that there exists a prime divisor $D$ such that $\dim \mathcal N_1(D,X)=1$: then, in \cite[Lemma 3.1]{CD}, the authors prove that either $\rho_X=1$, or $2 \leq \rho_X \leq 3$ and $X$ has some contraction onto a smooth Fano variety of lower Picard number. This contraction is either a blow-up or a conic bundle if $\rho_X=2$, and it is a conic bundle of relative Picard number 2 if $\rho_X=3$. Furthermore, in \cite[Theorem 3.8]{CD}, the authors give a general characterisation of $X$ when $\rho_X = 3$ (see Remark \ref{re1}).

Our goal in this paper is to classify all families of Fano fourfolds with $\rho_X=3$ containing a prime divisor $D$ with $\rho_D=1$, or more generally with $\dim \mathcal N_1(D,X)=1$. We compute the main numerical invariants of $X$, study the base locus of the linear system $|-K_X|$, and we also try to address the question of its rationality.

Recall that the index of a Fano variety, denoted by $i_X$, is the greatest positive integer $r$ such that the anticanonical divisor $-K_X$ is linearly equivalent to $rD$ for some ample Cartier divisor $D$ in $X$.

\begin{theorem}
\label{thm1}
There are $28$ families of smooth Fano fourfolds $X$ with $\rho_X=3$ having a prime divisor $D$ such that $\dim \mathcal N_1(D,X)=1$. They have index $i_X=1$.

Among them, $22$ families consist of rational varieties, and the very general variety of $4$ other families is not rational.

Furthermore, the linear system $|-K_X|$ is free, with the exception of two families where the base locus consists of either $1$ or $2$ points. In all cases, a general element of $|-K_X|$ is smooth.

The families with their numerical invariants and properties are described in Table $2$.
\end{theorem}

We also study their deformations in order to compute the dimension of the cohomology groups $H^0(\mathcal T_X)$ and $H^1(\mathcal T_X)$ of the tangent sheaf $\mathcal T_X$ of $X$. They correspond, respectively, to the dimension of the automorphism group ${\rm Aut}(X)$ and the dimension of the base of the Kuranishi family of a general member. We recall that deformations of smooth Fano varieties are unobstructed, since all obstruction classes of deformations of $X$ vanish. These classes are defined as elements of the cohomology group $H^2(\mathcal T_X)$, which vanishes as a consequence of the Nakano vanishing theorem \cite[Theorem 4.2.3]{Laz}.

The paper is organised as follows. In Section \ref{section2} we recall some results from \cite[Section 3]{CD}, and we use them to identify the list of 28 families of Fano fourfolds of Theorem \ref{thm1}. In Section \ref{section3} we compute the Hodge numbers, the anticanonical degree $K_X^4$, and the dimension of the space of global sections $H^0(\bigO X(-K_X))$ for all 28 families. In Section \ref{section4} we study the deformations of the varieties of Theorem \ref{thm1} in order to estimate $h^0(\mathcal T_X)$ and $h^1(\mathcal T_X)$. In Section \ref{section5} we study the base locus of the linear system $|-K_X|$, while in Section \ref{section6} we determine whether the Fano fourfolds under consideration are rational or not (when possible). In Section \ref{section7} we display the final tables with all the results of the paper.

Related results appear in \cite{Bat} and \cite{Tsu}: in the former the author provides a complete classification of toric Fano fourfolds, while in the latter the author classifies Fano varieties of dimension $\geq 3$ which admit a negative divisor isomorphic to the projective space. See Section \ref{section7} for the relations with these classifications.

Finally, we fix some notations. Let $X$ be a smooth, complex projective variety: we denote by
\begin{itemize}\setlength\itemsep{.25em}
\item $\mathcal N^1(X)$ the real vector space of divisors with real coefficients, modulo numerical equivalence. It is dual to $\mathcal N_1(X)$;
\item $\NE X \subset \mathcal N_1(X)$ the cone of effective 1-cycles;
\item $\Nef X \subset \mathcal N^1(X)$ the cone of nef divisors. It is dual to $\overline{\NE X}$, the closure of $\NE X$.
\end{itemize}
Recall that, for a smooth Fano variety, linear and numerical equivalence of divisors coincide, and we denote the latter by $\equiv$.

\bigskip

\textbf{Acknowledgements} We would like to thank Cinzia Casagrande for the continuous support received, and Angelo F. Lopez, Stéphane Druel, Eleonora A. Romano for their many comments and useful suggestions.

\section{The 28 families of Fano fourfolds}

\label{section2}
In this section we recall the needed results from \cite[Section 3]{CD}, and use them to determine the 28 families of Fano fourfolds that appear in Theorem \ref{thm1}.

\begin{remark}
\label{re1}
The following is a summary of \cite[Example 3.4 - Theorem 3.8]{CD}, when $n=4$.

Given a smooth Fano threefold $Z$, we begin by constructing from $Z$ a Fano fourfold $X$ with $\rho_X=3$, and containing a divisor with $\rho=1$. Then, by \cite[Theorem 3.8]{CD} we will see that all the varieties in Theorem \ref{thm1} are constructed in this way.

Let $Z$ be a smooth Fano threefold with Picard number $\rho_Z=1$ and index $i_Z$, which is the greatest positive integer $r$ such that $-K_Z\equiv rD$ for some ample divisor $D$ in $\Pic Z$. Also, let $\bigO Z(1)$ be the ample generator of $\Pic Z$, whose linear system $|\bigO Z(1)|$ is non-empty by applying the Riemann-Roch and Kodaira vanishing theorems. Then, take an effective divisor $H \in |\bigO Z(1)|$, so that $-K_Z \equiv i_Z H$. Moreover, fix integers $a \geq 0$ and $d \geq 1$, and assume that the linear system $|\bigO Z(d)|$ contains smooth surfaces. Then, fix such a smooth surface $A \in  |\bigO Z(d)|$.

Let $Y:=\Proj(\bigO Z \oplus \bigO Z(a))$ be the $\proj 1$-bundle with projection $\pi \colon Y \to Z$.
Take a section $G_Y$ with normal bundle $\normal {G_Y} Y \cong \bigO Z(-a)$, corresponding to a surjection $\bigO Z \oplus \bigO Z(a) \twoheadrightarrow \bigO Z$. If $a=0$, then $Y \cong Z \times \proj 1$ and take another section $\widehat G_Y$ disjoint from $G_Y$; if $a > 0$, take a section $\widehat G_Y$ with normal bundle $\normal {\widehat G_Y} Y \cong \bigO Z(a)$, which is also disjoint from $G_Y$, corresponding to a surjection $\bigO Z \oplus \bigO Z(a) \twoheadrightarrow \bigO Z(a)$.

Set $S:=\widehat G_Y \cap \pi^{-1}(A)$, and let $\sigma \colon X \to Y$ be the blow-up with centre $S$. Let us denote by $E$ the exceptional divisor and, by $G$ and $\widehat G$, the transforms of $G_Y$ and $\widehat G_Y$. Note that $G \cap \widehat G = \emptyset$, $\normal G X \cong \bigO Z (-a)$ and $\normal {\widehat G} X \cong \bigO Z (-(d-a))$.

Then $\varphi:=\pi \circ \sigma$ is a conic bundle, that is a fibre-type contraction whose generic fibre is a plane conic, and $\varphi$ admits another factorisation $\varphi:=\widehat \pi \circ \widehat \sigma$, where
\[\widehat \pi \colon \widehat Y := \Proj(\bigO Z \oplus \bigO Z(d-a)) \to Z\]
is a $\proj 1$-bundle, and $\widehat \sigma \colon X \to \widehat Y$ is the blow-up with centre $\widehat S:= \widehat \sigma (G) \cap \widehat \pi ^{-1}(A)$. Denote by $\widehat E$ the exceptional divisor. Also, $\widehat \sigma(G)$ and $\widehat \sigma(\widehat G)$ are disjoint sections of $\widehat \pi$, $\normal {\widehat \sigma (G)}{\widehat Y} \cong \bigO Z(d-a)$ and $\normal {\widehat \sigma (\widehat G)}{\widehat Y} \cong \bigO Z(a-d)$.

\begin{center}
\begin{tikzpicture}
 \matrix (m) [matrix of math nodes,row sep=4em,column sep=4em,minimum width=2em]
  {
     X & Y \\
     \widehat Y & Z \\  };
  \path[-stealth]
    (m-1-1) edge node [above] {$\sigma$} (m-1-2)
    (m-1-2) edge node [right] {$\pi$} (m-2-2)
    (m-1-1) edge node [left] {$\widehat \sigma$} (m-2-1)
    (m-2-1) edge node [below] {$\widehat \pi$} (m-2-2)
    (m-1-1) edge node [above] {$\varphi$} (m-2-2)
    ;
\end{tikzpicture}
\end{center}

Now let $C_Z$ be an irreducible curve in $Z$ such that $\bigO Z(1) \cdot C_Z$ is minimal, $C_G \subset G$ and $C_{\widehat G} \subset \widehat G$ be curves corresponding to $C_Z$, $F \subset E$ and $\widehat F \subset \widehat E$ be fibres of, respectively, $\sigma$ and $\widehat \sigma$. Then the cone of effective one-cycles $\NE X$ is closed and polyhedral, and is generated by the classes of:
\begin{itemize}\setlength\itemsep{.25em}
    \item $F$, $\widehat F$ and $C_{\widehat G}$, if $a=0$;
    \item $F$, $\widehat F$, $C_G$ and $C_{\widehat G}$, if $0<a<d$;
    \item $F$, $\widehat F$ and $C_G$, if $a \geq d$.
\end{itemize}
All the relevant intersection numbers are collected in \cite[Table 1]{CD}, and one can see that
\begin{equation}
\label{eqq1}
dG +a \widehat E \equiv d \widehat G + (d-a)E.
\end{equation}
Furthermore, by \cite[Remark 3.6]{CD}, $X$ is a Fano fourfold if and only if
\begin{equation*} 
    a \leq i_Z-1, \qquad d-a \leq i_Z-1.
\end{equation*}
Moreover, by \cite[Remark 3.7]{CD}, the pairs $(a,d)$ and $(d-a, d)$ give rise to isomorphic fourfolds when $d \geq a$, so we can assume that $a > d$ or $ 0 \leq a \leq \frac{d}{2}$.

Finally, by \cite[Theorem 3.8]{CD}, any Fano fourfold $X$ of Picard number $\rho_X=3$ which admits a prime divisor $D$ with $\dim \mathcal N_1(D,X)=1$ is isomorphic to one of the varieties $X$ constructed above.
Any such $X$, then, is given by a triplet $(Z, a, d)$, where $Z$ is a smooth Fano threefold with $\rho_Z=1$ and:
\begin{equation}  \label{eq3}
\begin{split}
d \geq 1, \qquad & a>d \text{ or } 0 \leq a \leq \frac{d}{2}, \\
a \leq i_Z -1, \qquad & d-a \leq i_Z -1;
\end{split}
\end{equation}
moreover $|dH|$ contains a smooth surface. The inequalities \eqref{eq3} imply that the index of the Fano threefold $Z$ cannot equal to $1$. Notice that the prime divisors $G$ and $\widehat G$ are isomorphic to $Z$, $\rho_G=\rho_{\widehat G}=1$, hence $\dim \mathcal N_1(G,X)=\dim \mathcal N_1(\widehat G,X)=1$. 
\end{remark}

\begin{remark}
Further observations:
\label{re2}
\begin{itemize}
    \item[(i)] Since $-K_X \cdot F=1$, we have that the index of $X$ is $i_X=1$.
    \item[(ii)] The classes of $\varphi ^*H$, $\widehat G$, $E$, and the classes of $\varphi ^*H$, $G$, $\widehat E$ are both a basis of $\mathcal N^1(X)$.
    Moreover,
    \begin{equation*}
        \alpha \varphi ^* H + \beta \widehat G + (\beta + \gamma)E \equiv  (\alpha + a\beta +d\gamma) \varphi ^* H + \beta G - \gamma \widehat E
    \end{equation*}
    for all $\alpha, \beta, \gamma \in \mathbb R$, and 
    \begin{equation*}
        -K_X \equiv (i_Z-a) \varphi ^* H + 2 \widehat G + E \equiv  (i_Z + a - d) \varphi ^* H + 2 G + \widehat E.
    \end{equation*}
    \item[(iii)] Under the conditions \eqref{eq3}, $Y$ and $\widehat Y$ are also Fano fourfolds by \cite[Example 3.3(2)]{Deb}. The cone $\Nef Y$ is generated by the classes of $\pi^*H$ 
    and $\widehat G_Y$, while $\Nef {\widehat Y}$ is generated by the classes of $\widehat \pi^*H$ and $\widehat \sigma (G)$. Note that the tautological line bundles $\bigO Y(1)$ and $\bigO {\widehat Y} (1)$ are isomorphic, respectively, to the line bundles $\bigO Y(\widehat G_Y)$ and $\bigO {\widehat Y} (\widehat \sigma (G))$.
\end{itemize}
\end{remark}

\begin{remark}
\label{re3}
For any irreducible curve $C \subset X$, we set $C^{\perp}:=\{D \in \mathcal N^1(X)| D\cdot C=0\}$. As $X$ is smooth and Fano, the nef cone $\Nef X$ is closed and polyhedral, and is generated by finitely many extremal rays. A section of $\Nef X$ is:\\
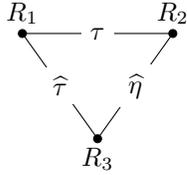
        \begin{minipage}{0.23\textwidth}
        \begin{tikzpicture}
        \draw (0,0) node[anchor=north]{$R_3$};
        \draw (1,1.4) node[anchor=south]{$R_2$};
        \draw (-1,1.4) node[anchor=south]{$R_1$};
        \draw[fill=black] (0,0) circle (1.5pt);
        \draw[fill=black] (1,1.4) circle (1.5pt);
        \draw[fill=black] (-1,1.4) circle (1.5pt);
        \draw (0,0) -- (1,1.4) node [midway, fill=white] {$\widehat \eta$};
        \draw (0,0) -- (-1,1.4) node [midway, fill=white] {$\widehat \tau$};
        \draw (1,1.4) -- (-1,1.4) node [midway, fill=white] {$\tau$};
        \end{tikzpicture}
        \end{minipage}
        \begin{minipage}{0.75\textwidth}
        $R_1= \mathbb R^+ [\varphi^* H]$, $R_2= \mathbb R^+ [G]$, and $R_3= \mathbb R^+ [G + \widehat E]$, if $a=0$;
        \end{minipage}\\
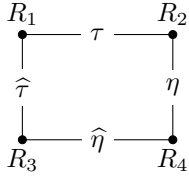
        \begin{minipage}{0.23\textwidth}
        \begin{tikzpicture}
        \draw (1,0) node[anchor=north]{$R_4$};
        \draw (-1,0) node[anchor=north]{$R_3$};
        \draw (1,1.4) node[anchor=south]{$R_2$};
        \draw (-1,1.4) node[anchor=south]{$R_1$};
        \draw[fill=black] (1,0) circle (1.5pt);
        \draw[fill=black] (-1,0) circle (1.5pt);
        \draw[fill=black] (-1,1.4) circle (1.5pt);
        \draw[fill=black] (1,1.4) circle (1.5pt);
        \draw (1,0) -- (-1,0) node [midway, fill=white] {$\widehat \eta$};
        \draw (-1,0) -- (-1,1.4) node [midway, fill=white] {$\widehat \tau$};
        \draw (1,1.4) -- (-1,1.4) node [midway, fill=white] {$\tau$};
        \draw (1,0) -- (1,1.4) node [midway, fill=white] {$\eta$};
        \end{tikzpicture}
        \end{minipage}
        \begin{minipage}{0.75\textwidth}
        $R_1= \mathbb R^+ [\varphi^* H]$, $R_2= \mathbb R^+ [a\varphi^*H +G]$, and $R_3= \mathbb R^+ [G + \widehat E]$, \\$R_4= \mathbb R^+ [dG + a\widehat E]$, if $0<a \leq \frac{d}{2}$;
        \end{minipage}\\
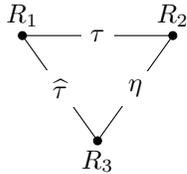
        \begin{minipage}{0.23\textwidth}
        \begin{tikzpicture}
        \draw (0,0) node[anchor=north]{$R_3$};
        \draw (1,1.4) node[anchor=south]{$R_2$};
        \draw (-1,1.4) node[anchor=south]{$R_1$};
        \draw[fill=black] (0,0) circle (1.5pt);
        \draw[fill=black] (1,1.4) circle (1.5pt);
        \draw[fill=black] (-1,1.4) circle (1.5pt);
        \draw (0,0) -- (1,1.4) node [midway, fill=white] {$\eta$};
        \draw (0,0) -- (-1,1.4) node [midway, fill=white] {$\widehat \tau$};
        \draw (1,1.4) -- (-1,1.4) node [midway, fill=white] {$\tau$};
        \end{tikzpicture}
        \end{minipage}
        \begin{minipage}{0.75\textwidth}
        $R_1= \mathbb R^+ [\varphi^* H]$, $R_2= \mathbb R^+ [a\varphi^*H +G]$, and $R_3= \mathbb R^+ [\widehat G]$, if $a > d$.
        \end{minipage}

Moreover, $\tau=F^{\perp} \cap \Nef X$, $\widehat \tau= {\widehat F}^{\perp} \cap \Nef X$, $\eta = {C_G}^{\perp} \cap \Nef X$, and $\widehat \eta = {C_{\widehat G}}^{\perp} \cap \Nef X$. One can check that all the contractions are divisorial, apart from the contraction given by the ray $R_1$, which is always of fibre type and coincides with the conic bundle $\varphi \colon X \to Z$, and the contraction given by the ray $R_2$ for $a=0$, which is again of fibre type and corresponds to a morphism $X \to \proj 1$; the contractions of $\tau$ and $\widehat \tau$ are $\sigma \colon X \to Y$ and $\widehat \sigma \colon X \to \widehat Y$, while the contractions of $\eta$ and $\widehat \eta$ send a divisor to a point, respectively $G$ and $\widehat G$.
\end{remark}

In \cite[Table \S 12.1]{IP} we find the list of all smooth Fano threefolds with Picard number $\rho_Z=1$. We are interested in those with index $i_Z \geq 2$, and we collect all the relevant varieties in Table 1 for the reader's convenience. The last two columns display, respectively, the base locus of the linear system $|H|$, and whether these varieties are rational ($+$) or not ($-$). We also include the dimension of the cohomology groups $H^0(\mathcal T_Z)$ and $H^1(\mathcal T_Z)$ of the tangent sheaf $\mathcal T_Z$. The references for $h^0(\mathcal T_Z)$ are \cite[Theorem 1.3]{PS} for $Z_1$, \cite[Theorem 3.4]{LP} for $Z_2$, \cite[Theorem 7.5]{PV} for $Z_5$, and \cite[Theorem 1.1, Theorem 1.3]{CPZ} for $Z_3$, $Z_4$ and $Z_6$ (see also \cite[Theorem 3.1]{Ben}). For similar results on the automorphism group of smooth Fano threefolds with $\rho_Z=1$, see also \cite[Theorem 1.1.2]{KPS}.

As for computing $h^1(\mathcal T_Z)$, by applying the Hirzebruch-Riemann-Roch theorem \cite[Appendix A, Theorem 4.1]{Har} to the sheaves $\mathcal T_Z$ and $\Omega^1_Z$, and the Nakano vanishing theorem \cite[Theorem 4.2.3]{Laz}, we get $$\chi(\mathcal T_Z)=h^0(\mathcal T_Z) - h^1(\mathcal T_Z)=-\frac{1}{2} K_Z^3 - \hodge 12Z - 17.$$

\begin{center}
\textbf{\large{Table 1 - Fano threefolds of Picard number 1 and index $\mathbf{i_Z \geq2}$}}
\end{center}
\begin{center}
\makebox[\textwidth]{
\begin{tabular}{ |m{0.55\textwidth}|c|c|c|c|c|c|c|c|}
\hline
$Z$ & $i_Z$ & $H^3$ & $-K_Z^3$ & $h^{1,2}$ & $h^0(\mathcal T_Z)$ & $h^1(\mathcal T_Z)$ & ${\rm Bs}|H|$ & R \\
\hline
\hline
$Z_1=$ a hypersurface of degree $6$ in the weighted projective space $\Proj(1,1,1,2,3)$ & $2$ & $1$ & $8$ & $21$ & 0 & 34 & $\{P_0\}$ & $-$\\ 
\hline
$Z_2=$ a cyclic cover of degree $2$ of $\proj 3$ ramified along a smooth surface of degree $4$ & $2$ & $2$ & $8\cdot2$ & $10$ & 0 & 19 & $\emptyset$ & $-$ \\ 
\hline
$Z_3=$ a smooth cubic in $\proj 4$ & $2$ & $3$ & $8\cdot3$ & $5$ & 0 & 10 & $\emptyset$ & $-$ \\ 
\hline
$Z_4=$ a smooth intersection of two quadrics in $\proj 5$ & $2$ & $4$ & $8\cdot4$ & $2$ & 0 & 3 & $\emptyset$ & $+$ \\
\hline
$Z_5 \subset \proj 6$, a section of the Grassmannian ${\rm Gr}(2,5) \subset \proj 9$ by a subspace of codimension $3$ & $2$ & $5$ & $8\cdot5$ & $0$ & 3 & 0 & $\emptyset$ & $+$ \\
\hline
$Z_6=$ a smooth quadric in $\proj 4$ & $3$ & $2$ & $27\cdot2$ & $0$ & 10 & 0 & $\emptyset$ & $+$ \\
\hline
$Z_7=$ $\proj 3$ & $4$ & $1$ & $64$ & $0$ & 15 & 0 & $\emptyset$ & $+$ \\
\hline
\end{tabular}
}
\end{center}

\begin{lm}
\label{lemma1}
For $i=2, \dots, 7$, the line bundle $\bigO {Z_i}(d)$ is globally generated for all $d \geq 1$, and $\bigO {Z_1}(d)$ is globally generated for all $d \geq 2$. Moreover, a general surface $A \in |\bigO {Z_1}(1)|$ is smooth. 
\end{lm}
\begin{proof}
Observe that, when $i \neq 1,2$, the ample generator $\bigO {Z_i}(1)$ of $\Pic {Z_i}$ corresponds to a hyperplane section in some projective space $\Proj$, thus it is very ample. This is clear for $i=3,4,6,7$, while it follows from \cite[Theorem 3.2.5(v)]{IP} that $\bigO {Z_5}(1)$ is the restriction to $Z_5$ of the very ample divisor on the Grassmannian ${\rm Gr}(2,5)$ defining the Pl{\"u}cker embedding ${\rm Gr}(2,5) \hookrightarrow \mathbb P^9$.
Furthermore, $\bigO {Z_2}(1)$ is globally generated since it is the pull-back of $\bigO {\proj 3} (1)$. Therefore, any positive multiple of the ample generator is globally generated.

As for $Z_1$, by \cite[Proposition 2.4.1(i), Theorem 2.4.5(i)]{IP} we have that $\bigO {Z_1}(2)$ is globally generated, while $\bigO {Z_1}(1)$ has a unique simple base point. Moreover, $\bigO {Z_1}(3)$ is very ample (see \cite[Table 1]{KPS}), so that $\bigO {Z_1}(d)$ is globally generated for all $d \geq 2$. Nonetheless, by \cite[Proposition 2.3.1]{IP}, a general surface $A$ in the linear system $|\bigO {Z_1}(1)|$ is smooth.\qedhere
\end{proof}

\begin{co}
\label{co1}
For every $Z_i$, $i=1, \dots, 7$, and for every choice of integers $a,d$ satisfying \eqref{eq3}, there exists a smooth Fano fourfold $X^i_{a,d}$ of Picard number 3 constructed as in Remark \ref{re1}. 

Conversely, if $X$ is a smooth Fano fourfold with $\rho_X=3$ and containing a prime divisor $D$ with $\dim \mathcal N_1(D,X)=1$, then $X$ is isomorphic to $X^i_{a,d}$ for some $(Z_i, a, d)$ as above. This gives rise to 28 families, see Table $2$.
\end{co}

\begin{proof}
Assume that the integers $a,d$ satisfy \eqref{eq3}. This yields $d \leq 2i_Z-2$. Therefore, by Lemma \ref{lemma1}, a general surface $A \in |\bigO {Z_i}(d)|$ is smooth for every $i=1, \dots, 7$, and so, given the triplet $(Z_i, a, d)$, we can construct a Fano fourfold $X:=X^i_{a,d}$ of Picard number 3. Notice that $X^i_{a,d}$ contains a prime divisor $G$ such that $\rho_G=1$, hence $\dim \mathcal N_1(G,X)=1$  (see Remark \ref{re1}).

Conversely, by \cite[Theorem 3.8]{CD}, any Fano fourfold $X$ of Picard number $\rho_X=3$ which admits a prime divisor $D$ with $\dim \mathcal N_1(D,X)=1$ is isomorphic to one of the $X^i_{a,d}$.
\end{proof}

Corollary \ref{co1} is the first step towards the proof of Theorem \ref{thm1}, as it shows that there are 28 choices of integers $a,d$ satisfying \eqref{eq3}, each of which determines a family of smooth Fano varieties, and that every smooth Fano fourfold of Picard number 3 containing a prime divisor of Picard number 1 appears in one of these families. We will later see in Section \ref{section4} that the members of each family are all deformation equivalent, and that distinct families do not have common members (see Corollary \ref{co4.2}).

\smallskip

\noindent \textbf{Notation.} If not specified otherwise, the varieties $Z$, $Y$, $X$, the corresponding maps and divisors will be the one described in Remark \ref{re1}. When needed, we will use the notation $Z_i$ for the varieties in Table 1, and $X^i_{a,d}$ for the variety given by the triplet $(Z_i, a, d)$.

\section{Numerical invariants}
\label{section3}

In this section we compute the numerical invariants of the Fano fourfolds in Corollary 	\ref{co1}; these invariants are listed in Table 2.

Our first goal is to compute the Hodge numbers. To do so, we will use the Hodge polynomial of a smooth variety $W$
\[e(W)= e(W) (u,v):= \sum_{p,q} \hodge pqW u^p v^q\]
and its properties with respect to $\proj n$-bundles and blow-ups.
Namely:
\begin{lm}
\label{ex4}
Let $W$ be a smooth variety, $\Proj (\mathcal E)$ a $\proj n$-bundle over $W$ and $\widetilde W$ the blow-up of $W$ along a subvariety $V$ of codimension $c$.
Then:
\begin{equation*}
\label{1}
    e(\Proj(\mathcal E))=e(W) \cdot e(\proj n)
\end{equation*}
and 
\begin{equation*}
\label{2}
    e(\widetilde W)=e(W) + e(V)\cdot [e(\proj {c-1})-1].
\end{equation*}
\end{lm}
\begin{proof}
By \cite[Introduction 0.1]{Che} and quick computations.\qedhere
\end{proof}

\smallskip

Since $Y$ and $X$ are Fano varieties of Picard number, respectively, 2 and 3, the only unknown Hodge numbers are $h^{1,2}$, $h^{1,3}$, and $h^{2,2}$.
We begin by computing the Hodge numbers of $A$, which is the smooth surface in $|\bigO Z(d)|$ fixed in Remark \ref{re1}. Recall that \eqref{eq3} yields $d \leq 2i_Z-2$, and that $H$ is an effective divisor of the linear system $|\bigO Z(1)|$.

\begin{lm}\label{hodge A}
Let $Z$ be a smooth Fano threefold with $\rho_X=1$ and $i_Z \geq 2$. Let $A$ be a smooth surface in $|\bigO Z(d)|$. Then all but the following Hodge numbers of $A$ vanish:
\begin{itemize}\setlength\itemsep{.25em}
\item $\hodge 00A=1$;
\item $\hodge 02A=1$,\, for $d=i_Z$;
\item $\hodge 02A=5$,\, for $d=4$ and $i_Z=3$ (i.e. $Z$ is a smooth quadric threefold);
\item $\hodge 02A=\binom{d-1}{3}$,\, for $d=5,6$ and $i_Z=4$ (i.e $Z \cong \mathbb P^3$);
\item $\hodge 11A=10+10\cdot\hodge 02A-d(d-i_Z)^2\delta$, where $\delta=H^3$.
\end{itemize}
\end{lm}
\begin{proof}
We first carry out the computation for $\hodge 02A$. By the exact sequence
\[0 \to \bigO Z(-d) \to \bigO Z \to \bigO A \to 0\]
we have $H^2(\bigO A) \cong H^3(\bigO Z(-d))$. Moreover, $h^3(\bigO Z(-d))=0$ for $d < i_Z$, and $h^3(\bigO Z(-d))=1$ for $d=i_Z$ by Kodaira vanishing.
This concludes the case $i_Z=2$. For $i_Z=3$ we have $1 \leq d \leq 4$ and $h^3(\bigO Z(-4))=h^4(\proj 4, \bigO {\proj 4}(-6))=5$;
for $i_Z=4$ we have $1 \leq d \leq 6$ and $h^3(\bigO {\proj 3}(-d))=\binom{d-1}{3}$ for $d=5,6$.

To compute $\hodge 11A$ it suffices to apply Noether's formula $\chi_{top}(A)=12\chi(\bigO A)-K_A^2$ \cite[Remark I.14]{Bea} to get
\begin{equation*}
    \hodge 11A=10+10\cdot \hodge 02A-d(d-i_Z)^2\delta
\end{equation*}
where $\delta=H^3$. Lastly, the vanishing of $\hodge 01A$ follows from the Lefschetz hyperplane theorem.
\end{proof}

\smallskip

By Lemma \ref{ex4} we get $e(Y)=e(Z) \cdot e(\proj 1)$ and $e(X)=e(Y)+e(A)\cdot(e(\proj 1)-1)$, thus the Hodge numbers of $Y$ and $X$ are:
\begin{equation*}
  \hodge pqY =
    \begin{cases}
      2 & (p,q)=(2,2)\\
      \hodge 12Z & (p,q)=(1,2)\\
      0 & (p,q)=(1,3)\\
    \end{cases}       
\end{equation*}
and
\begin{equation*}
  \hodge pqX =
    \begin{cases}
      \hodge 12Z & (p,q)=(1,2)\\
      \hodge 02A & (p,q)=(1,3)\\
      2 + \hodge 11A & (p,q)=(2,2) \ .\\
    \end{cases}
\end{equation*}
Notice that all the Hodge numbers of $X$ only depend on $d$ and $Z$, and not on $a$.

\bigskip

Our next goal is to compute the anticanonical degree $K_X^4$ and the dimension of the space of global sections $H^0(\bigO X(-K_X))$. We can do it simultaneously: the Hirzebruch-Riemann-Roch theorem \cite[Appendix A, Theorem 4.1]{Har} provides a formula to compute the Euler characteristic of a locally free sheaf on a smooth projective variety. We apply it to the anticanonical sheaf $\bigO W(-K_W)$ of a smooth fourfold $W$, yielding
\begin{equation}
\label{eq4}
    \chi(\bigO W(-K_W))=\chi(\bigO W) + \frac{1}{12}\Bigl(2K_W^4+K_W^2 \cdot c_2(W)\Bigr).
\end{equation}
Since $X$ is Fano, Kodaira vanishing implies $\chi(\bigO X(-K_X))=h^0(\bigO X(-K_X))$.

\smallskip

The following two standard computations provide the necessary tools.
\begin{pr}[\cite{CR}, Lemma 3.2]
\label{Prop.5}
Let $W$ be a smooth projective variety, $\dim W=4$, and let $\alpha \colon \widetilde W \to W$ be the blow-up of $W$ along a smooth irreducible surface V. Then:
\begin{itemize}\setlength\itemsep{.25em}
    \item[(i)] $K_{\widetilde W}^4=K_W^4-3(K_{W|V})^2-2K_V \cdot K_{W|V}+c_2(\normal VW)-K_V^2$;
    \item[(ii)] $K_{\widetilde W}^2 \cdot c_2(\widetilde W)=K_W^2 \cdot c_2(W)-12\chi(\bigO V)+2K_V^2-2K_V \cdot K_{W|V}-2c_2(\normal VW)$;
    \item[(iii)] $\chi(\bigO {\widetilde W}(-K_{\widetilde W}))=\chi(\bigO W(-K_W))-\chi(\bigO V)-\frac{1}{2}\Bigl((K_{W|V})^2+K_V \cdot K_{W|V}\Bigr)$.
\end{itemize}
\end{pr}

\begin{pr}
\label{Prop.6}
Let $W$ be a smooth projective variety of dimension $3$, and $\mathcal E$ a rank $2$ vector bundle on $W$.
Let $\beta \colon \Proj(\mathcal E) \to W$ be a $\proj 1$-bundle over $W$. Then:
\begin{itemize}\setlength\itemsep{.25em}
    \item[(i)] $K_{\Proj(\mathcal E)}^4=-8K_W \cdot c_1(\mathcal E)^2 +32K_W \cdot c_2(\mathcal E)-8K_W^3$;
    \item[(ii)] $K_{\Proj(\mathcal E)}^2 \cdot c_2(\Proj(\mathcal E))=-2K_W \cdot c_1(\mathcal E)^2+8K_W \cdot c_2(\mathcal E)-2K_W^3-4K_W \cdot c_2(W)$;
    \item[(iii)] $\chi(\bigO {\Proj(\mathcal E)}(-K_{\Proj(\mathcal E)}))= \chi(\bigO {\Proj(\mathcal E)}) + 6K_W \cdot c_2(\mathcal E) -\frac{1}{2} \Bigl(3K_W^3+3K_W \cdot c_1(\mathcal E)^2 \Bigr) -\frac{1}{3}K_W \cdot c_2(W)$.
\end{itemize}
\end{pr}
\begin{proof}
Let $D$ be the divisor associated to $\det (\mathcal E)$, and $\xi$ be the divisor associated to $\bigO{\Proj ({\mathcal E})}(1)$.
We have
$\sum_{i=0}^2 (-1)^i \beta ^*c_i(\mathcal E) \cdot \xi^{(2-i)}=0$, which yields
\[\xi^2=\beta^*D \cdot \xi-\beta^*c_2(\mathcal E).\]
Then:
\begin{itemize}\setlength\itemsep{.25em}
    \item $\beta^*(K_W + D)^4=0$;
    \item $\beta^*(K_W + D)^3 \cdot \xi=(K_W + D)^3$;
    \item $\beta^*(K_W + D)^2 \cdot \xi^2=\beta^*(K_W+D)^2 \cdot \Bigl(\beta^*D \cdot \xi-\beta^*c_2(\mathcal E)\Bigr)=(K_W + D)^2 \cdot D$;
    \item $\beta^*(K_W + D) \cdot \xi^3=\beta^*(K_W + D) \cdot \xi \cdot \Bigl(\beta^*D \cdot \xi-\beta^*c_2(\mathcal E)\Bigr)=  (K_W+D) \cdot D^2 -(K_W +D) \cdot c_2(\mathcal E)$;
    \item $\xi^4=\Bigl(\beta^*D \cdot \xi-\beta^*c_2(\mathcal E)\Bigr)^2=D^3-2D \cdot c_2(\mathcal E)$.
\end{itemize}
We recall that $K_{\Proj(\mathcal E)}=\beta^*(K_W+D)-2\xi$, therefore
\[K_{\Proj(\mathcal E)}^4=\sum_{i=0}^4 \binom{4}{i} \beta^*(K_W +D)^i \cdot (-2\xi)^{(4-i)},\]
and we obtain $(i)$.

To prove $(ii)$, we need to compute $c_2(\Proj(\mathcal E))$. By \cite[Example 3.2.11]{Ful}, we get
\[c_2(\Proj(\mathcal E))=\beta^* K_W \cdot (\beta^*D-2\xi)+\beta^*c_2(W).\]
So, $K_{\Proj(\mathcal E)}^2 \cdot c_2(\Proj(\mathcal E))=\Bigl(\beta^*(K_W+D) -2\xi \Bigr)^2 \cdot \Bigl(\beta^* K_W \cdot (\beta^*D-2\xi)+\beta^*c_2(W)\Bigr)$, which gives $(ii)$.
To get $(iii)$, just apply \eqref{eq4}.\qedhere
\end{proof}

We are now able to compute the numerical invariants of $X$ under consideration.
\begin{lm}
Let $X$ as in Corollary \ref{co1}, and let $\delta=H^3$. Then:
\begin{itemize}\setlength\itemsep{.25em}
	\item[(i)] $K_X^4=8\delta i_Z(a^2+i_Z^2) -3d\delta (a+i_Z)^2+2d\delta (a+i_Z)(d-i_Z) +ad^2\delta -d\delta (d-i_Z)^2$,
	\item[(ii)] $K_X^2 \cdot c_2(X)=84 + 2\delta i_Z(a^2+i_Z^2) -12 \hodge 02A +2d\delta (d-i_Z)(a+d)-2ad^2\delta$,
	\item[(iii)] $\chi (\bigO X(-K_X))=8+\frac{3}{2}\delta i_Z(a^2 + i_Z^2)-\hodge 02A -\frac{1}{2}d\delta (a+i_Z)(a-d+2i_Z)$.
\end{itemize}
\end{lm}
\begin{proof}
Recall that in the setting of Remark \ref{re1} $\mathcal E=\bigO Z \oplus \bigO Z(a)$, therefore $\det (\mathcal E)= \bigO Z(a)$, $c_1(\mathcal E)=aH$, and $c_2(\mathcal E)=0$. Since $Z$ is a Fano threefold, the Riemann-Roch Theorem \cite[Appendix A, Exercise 6.7]{Har} applied to the sheaf $\bigO Z$ yields $K_Z \cdot c_2(Z)=-24$. Moreover, the exact sequence
\[0 \to \normal S{\widehat G_Y} \to \normal{S}{Y} \to {\normal {\widehat G_Y}Y}_{|S} \to 0\]
corresponds to
\[0 \to \bigO A(dH_{|A}) \to \normal{S}{Y} \to \bigO A(aH_{|A}) \to 0,\]
and so 
\begin{enumerate}\setlength\itemsep{.25em}
    \item[(i)] $c_1(\normal{S}{Y})=(a+d)H_{|A}$, under the natural isomorphism $S \cong A$,
    \item[(ii)] $c_2(\normal{S}{Y})=ad^2\delta$.
\end{enumerate}
Furthermore, it is not difficult to verify that $K_{Y|S}=-(a+i_Z)H_{|A}$, $K_S=(d-i_Z)H_{|A}$, thus
\begin{gather*}
K_{Y|S}^2=d\delta (a+i_Z)^2, \quad K_S^2=d\delta (d-i_Z)^2, \\
K_S \cdot K_{Y|S}=-d\delta (a+i_Z)(d-i_Z).
\end{gather*}
Finally, Proposition \ref{Prop.5} yields 
\begin{itemize}\setlength\itemsep{.25em}
	\item[(i)] $K_Y^4=8\delta i_Z(a^2+i_Z^2)$,
	\item[(ii)] $K_Y^2 \cdot c_2(Y)=2\delta i_Z(a^2+i_Z^2) +96$,
	\item[(iii)] $\chi (\bigO Y(-K_Y))=9 + \frac{3}{2} \delta i_Z(a^2 + i_Z^2)$,
\end{itemize} 
and Proposition \ref{Prop.6} gives the statement.
\end{proof}
The explicit invariants for the varieties of Corollary \ref{co1} are collected in Table $2$, page \pageref{table2}.

\section{Deformations and automorphism groups}
\label{section4}
In this section we show that the varieties of Corollary \ref{co1} form indeed 28 distinct families of deformations, and we provide some partial results on the dimension of their Kuranishi family by giving an upper bound of $h^1(\mathcal T_X)$. This also leads to an upper bound of $h^0(\mathcal T_X)$, and precise results for some of the families. Recall that $h^0(\mathcal T_X)$ is the dimension of ${\rm Aut}(X)$, and that $h^2(\mathcal T_X)=0$ by the Nakano vanishing theorem \cite[Theorem 4.2.3]{Laz}. This implies that the base of the Kuranishi family of $X$ is smooth of dimension $h^1(\mathcal T_X)$, i.e.\ deformations of $X$ are unobstructed.

By a smooth family of Fano varieties we mean a smooth, projective morphism $f \colon \mathcal X \to S$ with connected fibres, and $S$ an irreducible quasi-projective variety, such that the anticanonical divisor $-K_{\mathcal X}$ is Cartier and $f$-ample.

\begin{remark}
The construction of $X$, as in Remark \ref{re1}, can be also done in families. Given $g \colon \mathcal Z \to B$ a smooth family of Fano threefolds with Picard number 1 and index $i_Z > 1$, and given $\bigO {\mathcal Z}(\mathcal H) \in \Pic {\mathcal Z}$ such that the restriction $\bigO {Z_b}(\mathcal H_{|_{Z_b}})$ to the fibre $Z_b$ is the generator of $\Pic {Z_b}$ for all $b \in B$, it is possible to construct a smooth family of Fano fourfolds $f \colon \mathcal X \to \widetilde B$, where $\widetilde B$ is the open subset of $\Proj_B \Bigl(\bigl(g_* \bigO {\mathcal Z}(d \mathcal H)\bigr)^{\vee}\Bigr)$ which parametrises the smooth surfaces $A_b \in |d\mathcal H_{|_{Z_b}}|$ for all $b \in B$, so that the fibre of $f$ over $\widetilde b \in \widetilde B$ is obtained from $(Z_b, A_b)$ as in Remark \ref{re1}.
\end{remark}

Let us fix a smooth Fano threefold $Z_i$, $i \in \{1, \dots, 7\}$, and integers $a, d$ satisfying \eqref{eq3}, so that the varieties $X^i_{a,d}$ of Corollary \ref{co1} are parametrised by couples $(Z, A)$, with $Z$ smooth and deformation equivalent to $Z_i$, and $A \in |\bigO Z(d)|$ a smooth surface. In this section we will use the notation $X(Z,A)$ to denote the Fano fourfold given by the couple $(Z,A)$.

Our first goal is to prove that if $f \colon \mathcal X \to S$ is a smooth family of Fano varieties, and if there exists $s_0 \in S$ such that its fibre $X_{s_0}$ is isomorphic to $X(Z_{s_0},A_{s_0})$, then all the fibres $X_s$ are isomorphic to $X^i_{a,d}$ for some $(Z_s, A_s)$, that is $X_s \cong X(Z_s, A_s)$. We are going to prove this by looking at the behaviour of the nef cone $\Nef {X_s}$, and the corresponding contractions, under the deformation $f$. Wiśniewski \cite[Theorem 1]{Wis2} proved that, in the case of a smooth family of Fano varieties, the nef cone is locally constant. Moreover, generalised results in the setting of singular varieties were proved in \cite{dFH} and \cite{CFST}.

\begin{pr}
\label{prop4.1}
Let $f \colon \mathcal X \to S$ be a smooth family of Fano varieties. Assume that there exists $s_0 \in S$ such that the fibre $X_{s_0}$ is isomorphic to $X^i_{a,d}$. Then, any fibre $X_s$ is one of the varieties of Corollary \ref{co1}, and belongs to the same family of $X_{s_0}$.
\end{pr}
\begin{proof}
\textbf{Step 1.} We can assume $\dim S=1$, and up to pull-back to the normalisation of $S$, we can assume $S$ smooth. Since the fibres of a smooth family are all diffeomorphic, we have that the second Betti number is constant on the fibres of $f$; moreover, for smooth Fano varieties, the second Betti number coincides with the rank of the Picard group. Therefore the fibres $X_s$ are smooth Fano fourfolds of Picard number 3, for all $s \in S$.

\textbf{Step 2.} In the following we adopt an argument in \cite{CFST} to show that the fibres of $f$ have the same type of contractions.

By \cite[Theorem 2.2]{CFST}, the monodromy action (see \cite[Section 3]{Voi}) of $\pi_1(S, s)$ on $H^2(X_s, \mathbb Q)$ is finite. As in the proof of \cite[Theorem 2.7]{CFST}, we may take a finite étale cover $g \colon V \to S$ and the pull-back $\mathcal X_V= \mathcal X \times_S V$ of $f$ via $g$. The resulting family $f_V \colon \mathcal X_V \to V$, which is still a deformation of ${X_{s_0}}$, now has trivial monodromy action on $H^2(X_t, \mathbb Q)$ for all $t \in V$. By \cite[Theorem 2.12]{CFST}, the restriction map
\[\mathcal N^1(\mathcal X_V) \to \mathcal N^1(X_t)\]
is surjective for all $t \in V$. As a consequence, the map 
\[\mathcal N^1(\mathcal X_V \slash V) \to \mathcal N^1(X_t)\]
is an isomorphism for all $t \in V$, where $\mathcal N^1(\mathcal X_V \slash V)$ is the real vector space of divisor with real coefficients, modulo numerical equivalence on the curves contracted by $f_V$.

Now, it follows from \cite[Proposition 1.3]{Wis1} that every elementary contraction of a fibre $X_t \to W_t$ extends to a relative elementary contraction $\mathcal X_V \to \mathcal W$, and conversely every relative elementary contraction $\mathcal X_V \to \mathcal W$ restricts to an elementary contraction $X_t \to W_t$ on every fibre. Therefore, \cite[Theorem 4.1]{dFH} implies that $X_{t_1} \to W_{t_1}$ is of fibre type (resp. divisorial, small) if and only if $X_{t_2} \to W_{t_2}$ of fibre type (resp. divisorial, small), for any $t_1, t_2 \in V$. By applying a similar technique as in the proof of \cite[Theorem 2.7]{CFST}, we see that the contractions on the fibres are deformation equivalent.

\textbf{Step 3.} By Step 2 and Remark \ref{re3}, we get that all the fibres $X_t$ of $f_V$ have an elementary divisorial contraction sending a divisor to a point, which implies that they all contain a divisor $G_t$ (the exceptional divisor of such contraction) with $\dim \mathcal N_1(G_t, X_t)=1$. Then, by \cite[Theorem 3.8]{CD}, each fibre $X_t$ is one of the varieties of Corollary \ref{co1}, and since the numerical invariants in Table 2 are constant under deformation and are different amongst the families, we can conclude that each $X_t$ is in the same family of $X_{s_0}$.
\end{proof}

\begin{co}
\label{co4.2}
The varieties of Corollary \ref{co1} form 28 families of deformations. They are all distinct, and do not have common members.
\end{co}

The next step is to understand when the varieties of the same family are isomorphic.

\begin{pr}
\label{prop4.2}
Let $X=X(Z,A)$ and $X'=X(Z', A')$. Then: $X \cong X'$ if and only if there exists an isomorphism $\psi \colon Z \to Z'$ such that $\psi(A)=A'$.
\end{pr}
\begin{proof}
Assume first that $X \cong X'$. Let $\varphi$ and ${\varphi}'$ be the conic bundles of Remark \ref{re1}, and let $\rho$ be the composition of ${\varphi}'$ with the isomorphism $X\cong X'$. The cones $\NE {\varphi}$ and $\NE {\rho}$, which are the convex cones in $\NE X$ generated by the curves contracted by $\varphi$ and $\rho$, must coincide, since they are both fibre-type contractions of $X$, and $X$ admits only one such contraction by Remark \ref{re3}. Therefore, by \cite[Proposition 1.14(b)]{Deb}, there exists a unique isomorphism $\psi \colon Z \to Z'$ which makes the following diagram commute.
\begin{center}
\begin{tikzpicture}
 \matrix (m) [matrix of math nodes,row sep=4.5em,column sep=4.5em,minimum width=2em]
  {
     X & X' \\
     Z & Z' \\ };
  \path[-stealth]
    (m-1-1) edge node [above] {$\cong$} (m-1-2)
    (m-1-1) edge node [left] {$\varphi$} (m-2-1)
    (m-1-2) edge node [right] {${\varphi}'$} (m-2-2)
    (m-2-1) edge node [below] {$\psi$} (m-2-2)
    (m-1-1) edge node [above] {$\rho$} (m-2-2)
    ;
\end{tikzpicture}
\end{center}
Since the diagram commutes, the discriminant divisor of $\varphi$ 
\[\Delta:=\{z \in Z| \varphi^{-1}(z) \text{ is singular}\}\]
maps via $\psi$ onto the discriminant divisor ${\Delta}'$ of ${\varphi}'$. In fact $\Delta = A$ and $\Delta' = A'$, and we get the statement,

Conversely, if there exists and isomorphism $\psi \colon Z \to Z'$ such that $\psi (A) = A'$, for smooth surfaces $A, A' \in |\bigO Z(d)|$, then $\psi$ lifts as an isomorphism between $X(Z, A)$ and $X(Z', A')$.
\end{proof}

At this point it is clear, by the previous results, that the dimension of the base of the Kuranishi family of $X^i_{a,d}$ is at most $h^1(\mathcal T_{Z_i}) + h^0(\bigO {Z_i}(d)) -1$.
\begin{pr}
\label{prop4.3}
Let $X=X^i_{a,d}$ be a Fano fourfold of Corollary \ref{co1}. Then:
\begin{equation}
\label{4.1}
h^1(\mathcal T_X) \leq h^1(\mathcal T_{Z_i}) + h^0(\bigO {Z_i}(d)) -1.
\end{equation}
Furthermore:
\begin{itemize}\setlength\itemsep{.25em}
\item[(i)] if $i=1,\dots,4$, then equality holds;
\item[(ii)] if $i=7$ and $d=1,2$, then $X$ is rigid, that is $h^1(\mathcal T_X)=0$.
\end{itemize}
\end{pr}
\begin{proof}
Inequality \eqref{4.1} follows from Propositions \ref{prop4.1} and \ref{prop4.2}. As for $(i)$, just observe that if $i=1,\dots,4$ then $h^0(\mathcal T_{Z_i})=0$ (see Table 1), so that ${\rm Aut}(Z_i)$ is finite. Therefore, for a given deformation $Z$ of $Z_i$ there are at most finitely many different choices of $A$ which yield isomorphic Fano fourfolds.

On the other hand, if $i=7$ and $d=1,2$ we are in the case of hyperplanes and smooth quadrics in $\proj 3$. Since they are all projectively equivalent to each other, it follows that $X^7_{a,1}$ and $X^7_{a,2}$ are rigid. This gives $(ii)$.
\end{proof}

We note that $\chi(\mathcal T_X)=h^0(\mathcal T_X)-h^1(\mathcal T_X)$ can be computed from the other invariants of $X$ (see Table 2) by the Hirzebruch-Riemann-Roch theorem, see for instance \cite[Lemma 6.25]{CCF}, which yields
\begin{equation*}
\chi(\mathcal T_X)=27 - 5h^0(\bigO X(-K_X)) + K_X^4 + 3 b_2(X) - \hodge 12X-\hodge 22X+ 3\hodge13X.
\end{equation*}
We can then estimate also $h^0(\mathcal T_X)$; we give the estimates in Table 3.

\section{Base locus of $-K_X$}
\label{section5}
In this section we study the base locus of the linear system $|-K_X|$, and show that it is empty for 26 of the 28 families of Corollary \ref{co1}, while it consists of at most two distinct points for the remaining cases. Moreover, we show that a general element of $|-K_X|$ is smooth. All the base loci are collected in Table 2.

To compute the base locus of $|-K_X|$, we give special effective divisors in the linear system.
We already know that $-K_X \equiv (i_Z-a) \varphi ^* H + 2 \widehat G + E \equiv  (i_Z + a - d) \varphi ^* H + 2 G + \widehat E$ (see Remark \ref{re2}).
Notice that $d \varphi^* H \equiv E + \widehat E$, and so, by \eqref{eqq1}, we have:
\begin{gather*}
G +a \varphi^*H \equiv \widehat G+E; \quad \widehat G +(d-a) \varphi^*H \equiv G + \widehat E
\end{gather*}
which yields
\begin{equation}
\label{eq4.1}
-K_X \equiv i_Z \varphi^*H + G+ \widehat G.
\end{equation}
\begin{pr}
\label{prop5.1}
Let $X$ be as in Corollary \ref{co1}. If $H$ is globally generated, then so is $-K_X$.
Otherwise, either $X \cong X^1_{0,1}$ and $|-K_X|$ has a unique base point, or $X \cong X^1_{1,2}$ and $|-K_X|$ has two base points. In all cases, a general element of $|-K_X|$ is smooth.
\end{pr}
\begin{proof}

\textbf{Step 1.} Assume that $H$ is globally generated. Then, by Remark \ref{re2}(ii) and \eqref{eq4.1}, we have
\[{\rm Bs}(|-K_X|) \subseteq \bigl(G \cup \widehat G \bigr) \cap \bigl(\widehat G \cup E \bigr) \cap \bigl(G \cup \widehat E \bigr)= \emptyset.\]
Therefore, $|-K_X|$ has no base point and $-K_X$ is globally generated.

Conversely, assume that $H$ is not globally generated. Then $Z=Z_1$, and either $a=0$ and $d=1$ or $a=1$ and $d=2$. The linear system $|H|$ has a unique simple base point $P_0$, while $|2H|$ is free; moreover, recall that in Remark \ref{re1} we fixed a smooth surface $A \in |dH|$, which exists for $Z=Z_1$ and $d=1,2$ (see Table 1 and Lemma \ref{lemma1}). 

\smallskip

\textbf{Step 2.} Let $Z=Z_1$, $a=0$ and $d=1$: then, the smooth surface $A \in |H|$ contains $P_0$, and the fibre $T:=\varphi^{-1}(P_0)$ over $P_0$ is equal to $F_0 \cup \widehat F_0$, with $F_0 \subset E$ and $\widehat F_0 \subset \widehat E$ irreducible rational curves. Therefore
\[{\rm Bs}(|-K_X|) \subseteq \bigl(G \cup \widehat G \bigr) \cap \bigl(\widehat G \cup E \bigr) \cap \bigl(T \cup G \cup \widehat E \bigr)= F_0 \cap \widehat G,\]
which is the point $Q_0 \in \widehat G$ corresponding to $P_0$ under the isomorphism $\widehat G \cong Z$. Now, observe that
\[\bigO {\widehat G}({-K_X}_{| \widehat G}) \cong \bigO Z(H),\]
and that, by computing the intersection number with the generators of $\NE X$, $-2K_X-\widehat G$ is an ample divisor (see \cite[Table 1]{CD}). This implies that, by Kodaira vanishing, $h^1(\bigO X(-K_X - \widehat G))=0$, and the restriction
\[H^0(\bigO X(-K_X)) \to H^0(\bigO {\widehat G}({-K_X}_{| \widehat G}))\]
is surjective. Thus, $Q_0$ is a base point for $|-K_X|$, and ${\rm Bs}(|-K_X|)=\{Q_0\}$. Notice that in the class of $2 \varphi^*H + G+ \widehat G$, which is numerically equivalent to $-K_X$ by \eqref{eq4.1}, we can choose an effective divisor so that it is smooth at $Q_0$: fix an effective divisor $D \equiv 2H$ so that $P_0 \notin D$, and $\varphi^* D \in |2 \varphi^*H|$ does not contain $Q_0$. Therefore $\varphi^*D + G + \widehat G$ is smooth at $Q_0$, thus a general effective divisor in $|-K_X|$ is smooth.

\smallskip

\textbf{Step 3.}
The last case to consider is $Z=Z_1$, $a=1$, $d=2$. Let $T$ be the fibre $\varphi^{-1}(P_0)$ over $P_0$: then, similarly as above,
\begin{gather*}
{\rm Bs}(|-K_X|) \subseteq \bigl(G \cup \widehat G \bigr) \cap \bigl(T \cup \widehat G \cup E \bigr) \cap \bigl(T \cup G \cup \widehat E \bigr)
=\bigl(T \cap G \bigr) \cup \bigl(T \cap \widehat G \bigr),
\end{gather*}
where $\{Q_1\}=T \cap G$ and $\{Q_2\}=T \cap \widehat G$ are the points corresponding to $P_0$ in $G$ and $\widehat G$. Again, similarly as above,
\[\bigO G({-K_X}_{| G}) \cong \bigO {\widehat G}({-K_X}_{| \widehat G}) \cong \bigO Z(H),\]
the divisors $-2K_X- G$ and $-2K_X-\widehat G$ are ample, the restrictions
\begin{gather*}
H^0(\bigO X(-K_X)) \to H^0(\bigO G({-K_X}_{|G})), \\
H^0(\bigO X(-K_X)) \to H^0(\bigO {\widehat G}({-K_X}_{| \widehat G}))
\end{gather*}
are surjective, thus $Q_1$ and $Q_2$ are base points for $|-K_X|$, and ${\rm Bs}(|-K_X|)=\{Q_1, Q_2\}$.

We now show that a general effective divisor of $|-K_X|$ is smooth. We have two possibilities: either $A \in |2H|$ contains $P_0$ or it does not. 
If we assume the latter, then the divisor $\varphi^*A + G+ \widehat G$ is smooth at $Q_1$ and $Q_2$. If we assume the former, then we can choose a divisor $D \in |2H|$ such that $P_0 \notin D$, which implies that the divisor $\varphi^*D + G + \widehat G \in |-K_X|$ is smooth at $Q_1$ and $Q_2$ (as in Step 2).
Thus the statement holds. \qedhere
\end{proof}

\section{Rationality}
\label{section6}
By construction, $X$ is birationally equivalent to $Z \times \proj 1$, and so we can study rationality on $X$ by what is known on $Z$.
If $i=4,5,6,7$, $Z_i$ is rational, while $Z_j$ is not rational for $j=1,2,3$ (see Table 1). So $X^i_{a,d}$ is rational for $i=4,5,6,7$.

The remaining cases have been studied in \cite{HT} with respect to stable rationality. For $i=1,2$, the very general $Z_i$ is not stably rational, which implies that $Z_i \times \proj 1$ is not rational. Therefore, the very general $X^i_{a,d}$ is not rational for $i=1,2$.
As for cubics in $\proj 4$, i.e. $Z_3$, stable rationality is still an open problem, and so we have no information on $X^3_{a,d}$.

\section{Conclusions and final tables}
\label{section7}
\begin{remark}
The proof of Theorem \ref{thm1} is a consequence of the results of the previous sections: Corollary \ref{co1} gives the 28 families, in Section \ref{section3} we compute the numerical invariants, Proposition \ref{prop5.1} provides the results on the base locus of $|-K_X|$, and in Section \ref{section6} we discuss rationality.
\end{remark}

\smallskip

Toric Fano fourfolds have been classified by Batyrev \cite{Bat}, and we observe that the varieties $X^i_{a,d}$ are toric when $Z=Z_7=\proj 3$ and $d=1$; this gives the three cases  $X^7_{0,1}$, $X^7_{2,1}$ and $X^7_{3,1}$. Keeping the notation in \cite{Bat}, they correspond, respectively, to the varieties $E_3$, $E_2$ and $E_1$.

Furthermore, Tsukioka \cite{Tsu} gave a classification for Fano varieties of dimension $\geq 3$ containing a negative divisor $E$ isomorphic to the projective space. When $n=4$, the varieties in \cite[Theorem 1, 2.(b)]{Tsu} are isomorphic to $X^7_{a,d}$ for all the admitted values $a,d$ given by \eqref{eq3}. Specifically, we can choose $E = \widehat G$ when $0 \leq a \leq \frac{d}{2}$, and $E = G$ when $a > d$ (see Remark \ref{re1} and Corollary \ref{co1}).

\begin{remark}
As a final note, we can ask whether the varieties $X^i_{a,d}$ of Theorem \ref{thm1} appear as fibres of elementary fibre-type contractions, that is if they are "fibre-like". This property has been introduced and studied in \cite{CFST}, where the authors give two criteria for fibre-likeness, one necessary and one sufficient. As a consequence of \cite[Corollary 3.9]{CFST}, by looking at the elementary contractions of $X^i_{a,d}$ (see Remark \ref{re3}) we see that, if $a \neq \frac{d}{2}$, then $X^i_{a,d}$ cannot be fibre-like.
\end{remark}

In Table $2$ we collect all the numerical invariants computed in Section \ref{section3}, where $h^0(-K_X):=h^0(\bigO X(-K_X))$, and in the last two columns we include the base locus of $|-K_X|$ and whether the variety is rational, not rational, toric, or unknown $(?)$.

\begin{center}\label{table2}
\textbf{\large{Table 2 - Fano fourfolds of Picard number 3 with a prime divisor of Picard number 1}}
\end{center}
\begin{center}
\makebox[\textwidth]{
\begin{tabular}{|c|c|c|>{\centering}m{0.1\textwidth}|c|c|c|>{\centering}m{0.12\textwidth}|m{0.31\textwidth}|}
\hline
$X^i_{a,d}$ & $K_X^4$ & $K_X^2 \cdot c_2(X)$ & $h^0(-K_X)$ & $h^{1,2}$ & $h^{1,3}$ & $h^{2,2}$ & ${\rm Bs}(|-K_X|)$ & \\
\hline
\hline
$X^1_{0,1}$ & 47 & 98 & 17 & 21 & 0 & 11 & $\{Q_0\}$ & the very general is not rational \\
\hline
$X^1_{1,2}$ & 30 & 84 & 13 & 21 & 1 & 22 & $\{Q_1, Q_2\}$ & the very general is not rational \\
\hline
$X^2_{0,1}$ & 94 & 112 & 26 & 10 & 0 & 10 & $\emptyset$ & the very general is not rational \\
\hline
$X^2_{1,2}$ & 60 & 96 & 19 & 10 & 1 & 22 & $\emptyset$ & the very general is not rational \\
\hline
$X^3_{0,1}$ & 141 & 126 & 35 & 5 & 0 & 9 & $\emptyset$ & $?$ \\
\hline
$X^3_{1,2}$ & 90 & 108 & 25 & 5 & 1 & 22 & $\emptyset$ & $?$ \\
\hline
$X^4_{0,1}$ & 188 & 140 & 44 & 2 & 0 & 8 & $\emptyset$ & rational \\
\hline
$X^4_{1,2}$ & 120 & 120 & 31 & 2 & 1 & 22 & $\emptyset$ & rational \\
\hline
$X^5_{0,1}$ & 235 & 154 & 53 & 0 & 0 & 7 & $\emptyset$ & rational \\
\hline
$X^5_{1,2}$ & 150 & 132 & 37 & 0 & 1 & 22 & $\emptyset$ & rational \\
\hline
$X^6_{0,1}$ & 346 & 184 & 74 & 0 & 0 & 4 & $\emptyset$ & rational \\
\hline
$X^6_{0,2}$ & 296 & 176 & 65 & 0 & 0 & 8 & $\emptyset$ & rational \\
\hline
$X^6_{1,2}$ & 260 & 164 & 58 & 0 & 0 & 8 & $\emptyset$ & rational \\
\hline
$X^6_{1,3}$ & 210 & 156 & 49 & 0 & 1 & 22 & $\emptyset$ & rational \\
\hline
$X^6_{2,1}$ & 430 & 208 & 90 & 0 & 0 & 4 & $\emptyset$ & rational \\
\hline
$X^6_{2,4}$ & 160 & 148 & 40 & 0 & 5 & 54 & $\emptyset$ & rational \\
\hline
$X^7_{0,1}$ & 431 & 206 & 90 & 0 & 0 & 3 & $\emptyset$ & toric \\
\hline
$X^7_{0,2}$ & 376 & 196 & 80 & 0 & 0 & 4 & $\emptyset$ & rational \\
\hline
$X^7_{0,3}$ & 341 & 194 & 74 & 0 & 0 & 9 & $\emptyset$ & rational \\
\hline
$X^7_{1,2}$ & 350 & 188 & 75 & 0 & 0 & 4 & $\emptyset$ & rational \\
\hline
$X^7_{1,3}$ & 295 & 178 & 65 & 0 & 0 & 9 & $\emptyset$ & rational \\
\hline
\end{tabular}
}
\end{center}
\begin{center}
\makebox[\textwidth]{
\begin{tabular}{|c|c|c|>{\centering}m{0.1\textwidth}|c|c|c|>{\centering}m{0.12\textwidth}|m{0.31\textwidth}|}
\hline
$X^i_{a,d}$ & $K_X^4$ & $K_X^2 \cdot c_2(X)$ & $h^0(-K_X)$ & $h^{1,2}$ & $h^{1,3}$ & $h^{2,2}$ & ${\rm Bs}(|-K_X|)$ & \\
\hline
\hline
$X^7_{1,4}$ & 260 & 176 & 59 & 0 & 1 & 22 & $\emptyset$ & rational \\
\hline
$X^7_{2,1}$ & 489 & 222 & 101 & 0 & 0 & 3 & $\emptyset$ & toric \\
\hline
$X^7_{2,4}$ & 240 & 168 & 55 & 0 & 1 & 22 & $\emptyset$ & rational \\
\hline
$X^7_{2,5}$ & 205 & 166 & 49 & 0 & 4 & 47 & $\emptyset$ & rational \\
\hline
$X^7_{3,1}$ & 605 & 254 & 123 & 0 & 0 & 3 & $\emptyset$ & toric \\
\hline
$X^7_{3,2}$ & 454 & 220 & 95 & 0 & 0 & 4 & $\emptyset$ & rational \\
\hline
$X^7_{3,6}$ & 170 & 164 & 43 & 0 & 10 & 88 & $\emptyset$ & rational \\
\hline
\end{tabular}
}
\end{center}

\bigskip

In Table 3 we display the explicit invariants of Section \ref{section4}. To do so, we are left to compute $h^0(\bigO {Z}(d))$, as the rest of the ingredients are already in Table 1 and Table 2. The Riemann-Roch and Kodaira vanishing theorems yield
\[\chi(\bigO {Z}(d))=h^0(\bigO {Z}(d))=1 + \frac{2d}{i_Z}+\frac{d \delta}{12}(i_Z^2+3d i_Z+2d^2).\]
\begin{center}
\textbf{\large{Table 3 - Cohomology groups of the tangent sheaf}}
\end{center}
\begin{center}
\begin{tabular}{|c|c|c|c|}
\hline
$X^i_{a,d}$ & $h^0(\mathcal T_X)$ & $h^1(\mathcal T_X)$ & $\chi(\mathcal T_X)$ \\
\hline
\hline
$X^1_{0,1}$ & 2 & 36 & -34 \\
\hline
$X^1_{1,2}$ & 1 & 40 & -39 \\
\hline
$X^2_{0,1}$ & 2 & 22 & -20 \\
\hline
$X^2_{1,2}$ & 1 & 29 & -28 \\
\hline
$X^3_{0,1}$ & 2 & 14 & -12 \\
\hline
$X^3_{1,2}$ & 1 & 24 & -23 \\
\hline
$X^4_{0,1}$ & 2 & 8 & -6 \\
\hline
$X^4_{1,2}$ & 1 & 21 & -20 \\
\hline
$X^5_{0,1}$ & $\leq 5$ & $\leq 6$ & -1 \\
\hline
$X^5_{1,2}$ & $\leq 4$ & $\leq 22$ & -18 \\
\hline
$X^6_{0,1}$ & $\leq 12$ & $\leq 4$ & 8 \\
\hline
$X^6_{0,2}$ & $\leq 12$ & $\leq 13$ & -1 \\
\hline
$X^6_{1,2}$ & $\leq 11$ & $\leq 13$ & -2 \\
\hline
$X^6_{1,3}$ & $\leq 11$ & $\leq 29$ & -18 \\
\hline
\end{tabular}
\begin{tabular}{|c|c|c|c|}
\hline
$X^i_{a,d}$ & $h^0(\mathcal T_X)$ & $h^1(\mathcal T_X)$ & $\chi(\mathcal T_X)$ \\
\hline
\hline
$X^6_{2,1}$ & $\leq 16$ & $\leq 4$ & 12 \\
\hline
$X^6_{2,4}$ & $\leq 11$ & $\leq 54$ & -43 \\
\hline
$X^7_{0,1}$ & 14 & 0 & 14 \\
\hline
$X^7_{0,2}$ & 8 & 0 & 8 \\
\hline
$X^7_{0,3}$ & $\leq 17$ & $\leq 19$ & -2 \\
\hline
$X^7_{1,2}$ & 7 & 0 & 7 \\
\hline
$X^7_{1,3}$ & $\leq 16$ & $\leq 19$ & -3\\
\hline
$X^7_{1,4}$ & $\leq 16$ & $\leq 34$ & -18 \\
\hline
$X^7_{2,1}$ & 17 & 0 & 17 \\
\hline
$X^7_{2,4}$ & $\leq 16$ & $\leq 34$ & -18 \\
\hline
$X^7_{2,5}$ & $\leq 16$ & $\leq 55$ & -39 \\
\hline
$X^7_{3,1}$ & 23 & 0 & 23 \\
\hline
$X^7_{3,2}$ & 11 & 0 & 11 \\
\hline
$X^7_{3,6}$ & $\leq 16$ & $\leq 83$ & -67 \\
\hline
\end{tabular}
\end{center}


\begin{thebibliography}{CFST}

\bibitem[Bat]{Bat}
V. V. Batyrev.
\textit{On the classification of toric Fano 4-folds}.
J. Math. Sci. (New York) 94 (1999), 1021–1050.

\bibitem[Bea]{Bea}
A. Beauville.
\textit{Complex algebraic surfaces}.
London Mathematical Society Student Texts, Vol. 34, Cambridge University Press, 1996.

\bibitem[Ben]{Ben}
O. Benoist.
\textit{Séparation et propriété de Deligne-Mumford des champs de modules d’intersections complètes lisses}.
J. Lond. Math. Soc. (2) 87 (2013), no. 1, 138–156.

\bibitem[CCF]{CCF}
C. Casagrande, G. Codogni, A. Fanelli.
\textit{The blow-up of $\proj 4$ at 8 points and its Fano model, via vector bundles on a del Pezzo surface}.
Rev. Mat. Complut. 32 (2019), no. 2, 475–529.

\bibitem[CD]{CD}
C. Casagrande, S. Druel.
\textit{Locally unsplit families of rational curves of large anticanonical degree on Fano manifolds}.
Int. Math. Res. Not. IMRN 2015, no. 21, 10756–10800.

\bibitem[CFST]{CFST}
G. Codogni, A. Fanelli, R. Svaldi, L. Tasin.
\textit{Fano varieties in Mori fibre spaces}.
Int. Math. Res. Not. IMRN 2016, no. 7, 2026–2067.

\bibitem[Che]{Che}
J. Cheah.
\textit{The Hodge polynomial of the Fulton-MacPherson compactification of configuration spaces}.
Amer. J. Math. 118 (1996), no. 5, 963–977.

\bibitem[CPZ]{CPZ}
X. Chen, X. Pan, D. Zhang.
\textit{Automorphism and cohomology II: complete intersections}.
Preprint arXiv:1511.07906v6.

\bibitem[CR]{CR}
C. Casagrande, E. A. Romano.
\textit{Classification of Fano 4-folds with Lefschetz defect 3 and Picard number 5}.
J. Pure Appl. Algebra 226 (2022), no. 3, 13 pp.

\bibitem[Deb]{Deb}
O. Debarre.
\textit{Higher-dimensional algebraic geometry}.
Springer-Verlag, New York, 2001.

\bibitem[dFH]{dFH}
T. de Fernex, C. D. Hacon.
\textit{Deformations of canonical pairs and Fano varieties}.
J. Reine Angew. Math. 651 (2011), 97–126.

\bibitem[Ful]{Ful}
W. Fulton.
\textit{Intersection theory - Second edition}.
Ergebnisse der Mathematik und ihrer Grenzgebiete. 3. Folge. A Series of Modern Surveys in Mathematics, 2. Springer-Verlag, Berlin, 1998.

\bibitem[Har]{Har}
R. Hartshorne.
\textit{Algebraic geometry}.
Graduate Texts in Math., Vol. 52, Springer-Verlag, New York/Berlin, 1977.

\bibitem[HT]{HT}
B. Hassett, Y. Tschinkel.
\textit{On stable rationality of Fano threefolds and del Pezzo fibrations}.
J. Reine Angew. Math. 751 (2019), 275–287.

\bibitem[IP]{IP}
V. A. Iskovskikh, Y. G. Prokhorov.
\textit{Algebraic geometry V - Fano varieties}.
Encyclopaedia Math. Sci. vol. 47, Springer-Verlag, 1999.

\bibitem[KPS]{KPS}
A. G. Kuznetsov, Y. G. Prokhorov, C. A. Shramov.
\textit{Hilbert schemes of lines and conics and automorphism groups of Fano threefolds}.
Jpn. J. Math. 13 (2018), no. 1, 109–185.

\bibitem[Laz]{Laz}
R. Lazarsfeld.
\textit{Positivity in algebraic geometry I - Classical setting: line bundles and linear series}.
Springer, 2004.

\bibitem[LP]{LP}
R. Lyu, X. Pan.
\textit{Remarks on automorphism and cohomology of finite cyclic coverings of projective spaces}.
Math. Res. Lett. 28 (2021), no. 3, 785–822.

\bibitem[PS]{PS}
V. V. Przyjalkowski, C. A. Shramov.
\textit{Automorphisms of weighted complete intersections}.
Proc. Steklov Inst. Math. 307 (2019), no. 1, 198–209.

\bibitem[PV]{PV}
J. Piontkowski, A. Van de Ven.
\textit{The automorphism group of linear sections of the Grassmannians $\mathbb G(1,N)$}.
Doc. Math. 4 (1999), 623–664.

\bibitem[Tsu]{Tsu}
T. Tsukioka.
\textit{Classification of Fano manifolds containing a negative divisor isomorphic to projective space}.
Geom. Dedicata 123 (2006), 179–186.

\bibitem[Voi]{Voi}
C. Voisin.
\textit{Hodge theory and complex algebraic geometry II}.
Camb. Stud. Adv. Math. 77, Cambridge Univ. Press 2003.

\bibitem[Wiś1]{Wis1}
J. A. Wiśniewski.
\textit{On deformation of nef values}.
Duke Math. J. 64 (1991), 325–332.

\bibitem[Wiś2]{Wis2}
J. A. Wiśniewski.
\textit{Rigidity of the Mori cone for Fano manifolds}.
Bull. Lond. Math. Soc. 41 (2009), no. 5, 779–781.

\end{thebibliography}
\end{document}